\documentclass[12pt]{article}

\usepackage{amsmath,amsthm,amsfonts,latexsym,amsopn,verbatim,amscd,amssymb}

\theoremstyle{plain}
\newtheorem{theorem}{Theorem}[section]
\newtheorem{lemma}[theorem]{Lemma}
\newtheorem{proposition}[theorem]{Proposition}
\newtheorem{corollary}[theorem]{Corollary}
\theoremstyle{definition}
\numberwithin{equation}{section}

\DeclareMathOperator{\spec}{Spec}

\newcommand{\bnum}{\begin{enumerate}}
\newcommand{\enum}{\end{enumerate}}

\begin{document}

%\title{SOME EQUIDISTRIBUTED AND MEASURE PRESERVING MAPS IN  FINITE  GROUP }
\title{Spectrum of commuting  graphs of some classes of finite groups }
\author{Jutirekha Dutta  and Rajat Kanti Nath\footnote{Corresponding author}}
\date{}
\maketitle
\begin{center}\small{ Department of Mathematical Sciences,\\ Tezpur
University,  Napaam-784028, Sonitpur, Assam, India.\\
Emails: jutirekhadutta@yahoo.com, rajatkantinath@yahoo.com}
\end{center}

%\thanks{}

\smallskip

\noindent {\small{\textbf{Abstract:}  In  this paper, we initiate the study of spectrum of the commuting graphs of finite non-abelian groups. We first compute the spectrum of this graph for several classes of finite groups, in particular AC-groups. We  show that  the  commuting graphs  of  finite non-abelian AC-groups are integral. We also  show that the  commuting graph of a finite non-abelian group $G$ is integral if $G$ is not isomorphic to the symmetric group of degree $4$ and the  commuting graph of $G$ is planar. Further it is shown that   the  commuting graph of $G$ is integral if the  commuting graph of $G$ is toroidal.

% if a finite non-abelian group $G$ is not isomorphic to the symmetric group of degree $4$ and the  commuting graph of $G$ is planar then the  commuting graph of $G$ is integral. We also show that   the  commuting graph of $G$ is integral if it is toroidal.}

%\bigskip


\section{Introduction} \label{S:intro}
Let $G$ be a finite group with centre $Z(G)$. The commuting graph of a non-abelian group $G$, denoted by $\Gamma_G$, is a simple undirected graph whose vertex set is $G\setminus Z(G)$, and two vertices $x$ and $y$ are adjacent if and only if $xy = yx$.   Various   aspects of commuting graphs of different finite groups can be found in  \cite{amr06,bbhR09,iJ07,iJ08,mP13,par13}.  In  this paper, we initiate the study of spectrum of commuting graphs of finite non-abelian groups.  Recall that the spectrum of a graph  ${\mathcal{G}}$ denoted by $\spec({\mathcal{G}})$ is the set $\{\lambda_1^{k_1}, \lambda_2^{k_2},$ $\dots, \lambda_n^{k_n}\}$, where $\lambda_1,  \lambda_2, \dots, \lambda_n$ are the eigenvalues of the adjacency matrix of ${\mathcal{G}}$ with multiplicities $k_1, k_2, \dots, k_n$ respectively. A graph ${\mathcal{G}}$ is called integral if $\spec({\mathcal{G}})$ contains only integers. It is well known that the complete graph $K_n$ on $n$ vertices is integral. Moreover, if ${\mathcal{G}}$ is the disjoint union of some complete graphs then also it is integral.     The notion of integral graph was introduced by  Harary and  Schwenk \cite{hS74} in the year 1974. A very impressive survey on integral graphs can be found in \cite{bCrS03}.

%  In particular, we compute the spectrum of the commuting graphs  of  the quasidihedral groups, general  linear groups,   some projective special linear groups, groups  with central factors   isomorphic to  $\langle a, b : a^5 = b^4 = 1, b^{-1}ab = a^2 \rangle$ and the groups constructed by Hanaki in \cite{Han96}.   

% Since then many mathematicians have constructed different types of integral graphs, see \cite{aV09, Br08,bH14, iV07, So07, saF07, wL05, wLH05}. 
 
%Ahmadi et. al noted  that   integral graphs have some interest for designing the network topology of perfect state transfer networks, see  \cite{anb09} and the references there in. 

%In recent years, some authors have considered graphs whose vertex sets are   algebraic structures and obtained several conditions under which the graphs become integral  (see for example \cite{aV09, al12, Nath15, Nd15}). 

%In particular, the authors have shown  that the  commuting graph of a finite non-abelian group $G$ is  integral if $G/Z(G)$ is isomorphic to ${\mathbb{Z}}_p \times {\mathbb{Z}}_p$  or $D_{2m}$,  where $p$ is a prime integer and $D_{2m}$ is the dihedral group of order $2m$ (see \cite{Nath15}).

  We observe that the commuting graph of a non abelian finite AC-group is disjoint union of some complete graphs. Therefore, commuting graphs of such groups are integral.  
%Our computation reveals that the commuting graphs of these groups are integral. 
In general it is difficult to classify all  finite non-abelian groups whose commuting graphs are integral. As applications of our results together with some other known results, in  Section 3, we show that the  commuting graph of a finite non-abelian group $G$ is  integral if $G$ is not isomorphic to $S_4$, the symmetric group of degree $4$, and the commuting graph of $G$ is planar. We also show that the  commuting graph of a finite non-abelian group $G$ is  integral if the commuting graph of $G$   is toroidal. Recall that the genus of a graph is the smallest non-negative integer $n$ such that the graph can be embedded on the surface obtained by attaching $n$ handles to a sphere. A graph is said to be planar or toroidal if the genus of the graph is zero or one respectively. It is worth mentioning that  Afkhami et al. \cite{AF14} and Das et al. \cite{das13} have classified all finite non-abelian groups whose commuting graphs are planar or toroidal recently.
  
% For a finite group $G$, the set $C_G(x) = \{y \in G : xy = yx\}$ is called the centralizer of an element $x \in G$. Let $|\cent(G)| = |\{C_G(x) : x \in G\}|$, that is the number of distinct centralizers in $G$. A group $G$ is called an $n$-centralizer group if $|\cent(G)| = n$. In \cite{bG94}, Belcastro and  Sherman  characterized $n$-centralizer groups for $n = 4, 5$. In this paper, we  show  that the commuting graphs of  $4, 5$-centralizer groups are integral,  as a consequence of our results. More generally, we show that the commuting graph of  a $p + 2$-centralizer $p$-group is integral for any prime $p$.

\section{Computing spectrum}

% The following result  obtained in \cite{Nath15} is used extensively in this section. 

It is well known that the complete graph $K_n$ on $n$ vertices is integral and $\spec(K_n)$ is given by $\{(-1)^{n - 1}, (n - 1)^1\}$. Further, if $\mathcal{G} = K_{m_1}\sqcup K_{m_2}\sqcup\cdots  \sqcup K_{m_l}$, where $K_{m_i}$ are complete graphs on $m_i$ vertices for $1 \leq i \leq l$, then 
\begin{equation} \label{prethm1}
\spec(\mathcal{G}) = \{(-1)^{\underset{i = 1}{\overset{l}{\sum}}m_i - l},\, (m_1 - 1)^1,\, (m_2 - 1)^1,\, \dots,\, (m_l - 1)^1\}.
\end{equation}

\noindent If $m_1 = m_2 = \cdots = m_l = m$  then we write $\mathcal{G} = lK_{m}$ and in that case  $\spec(\mathcal{G}) = \{(-1)^{l(m-1)}, (m - 1)^l\}$.

In this section, we compute the spectrum of the commuting graphs of different families of finite non-abelian AC-groups.  A group $G$ is called an AC-group if $C_G(x)$ is abelian for all $x \in G\setminus Z(G)$.  Various aspects of AC-groups can be found in \cite{Ab06,das13,Roc75}.
The following lemma plays an important role in  computing  spectrum of commuting graphs of AC-groups. 

\begin{lemma}\label{AC-Lem}
Let $G$ be a finite non-abelian   AC-group.  Then    the  commuting graph of   $G$ is given by
\[
\Gamma_G =  \overset{n}{\underset{i = 1}{\sqcup}}K_{|X_i| - |Z(G)|}
\]
where $X_1,\dots, X_n$ are the distinct centralizers of non-central elements of $G$.
\end{lemma}
\begin{proof}
Let $G$ be a finite non-abelian   AC-group and $X_1,\dots, X_n$ be the distinct centralizers of non-central elements of $G$. Let $X_i = C_G(x_i)$ where $x_i \in G \setminus Z(G)$ and $1\leq i \leq n$. Let $x, y \in X_i \setminus Z(G)$  for some $i$ and $x \ne y$  then, since $G$ an 
AC-group, there is an edge between $x$ and $y$ in the commuting graph of $G$.  Suppose that $x \in (X_i\cap X_j)\setminus Z(G)$ for some $1\leq i \ne j \leq n$. Then  $[x, x_i] = 1$ and $[x, x_j] = 1$.  Hence, by Lemma 3.6 of \cite{Ab06} we have $C_G(x) = C_G(x_i) = C_G(x_j)$, a contradiction. Therefore, $X_i\cap X_j = Z(G)$ for any $1\leq i \ne j \leq n$. This shows that  $\Gamma_G = \overset{n}{\underset{i = 1}{\sqcup}}K_{|X_i|-|Z(G)|}$.
\end{proof}

\begin{theorem}\label{AC-group}
Let $G$ be a finite non-abelian   AC-group.  Then    the spectrum of the commuting graph of   $G$ is given by
\[
 \{(-1)^{\overset{n}{\underset{i = 1}{\sum}}|X_i| - n(|Z(G)| + 1) }, (|X_1| - |Z(G)| - 1)^1,\dots, (|X_n| - |Z(G)| - 1)^1\} 
\]
where $X_1,\dots, X_n$ are the distinct centralizers of non-central elements of $G$.
\end{theorem}

\begin{proof}
The proof follows from Lemma \ref{AC-Lem} and    \eqref{prethm1}.
\end{proof}

%\begin{corollary} 
%Let $G$ and $H$ be two finite non-abelian  AC-groups. If $X_1,\dots, X_m$ and $Y_1,\dots, Y_n$  are the distinct centralizers of non-central elements of $G$ and $H$ respectively then the spectrum of the commuting graph of   $G\times H$ is given by
%\begin{align*}
% \{ (-1)^{\underset{i, j}{\sum}|X_i||Y_j| - mn(|Z(G)||Z(H)| + 1)}, (|X_i||Y_j| - |Z(G)|&|Z(H)| - 1)^1:\\
%&  1\leq i\leq m \text{ and } 1\leq j\leq n\} 
%\end{align*} 
%\end{corollary}

%\begin{proof}
%We have $Z(G\times H) = Z(G)\times Z(H)$. The distinct centralizers of non-central elements of $G\times H$ are given by $\{X_i\times Y_j : 1\leq i\leq m \text{ and } 1\leq j\leq n\}$.  Therefore, $G \times H$ is an  AC-group.   Hence, the result follows from Theorem \ref{AC-group}. 
%\end{proof}

\begin{corollary}\label{AC-cor}
Let $G$ be a finite non-abelian  AC-group and $A$ be any finite abelian group.  Then    the spectrum of the commuting graph of   $G\times A$ is given by
\begin{align*}
 \{ (-1)^{\overset{n}{\underset{i = 1}{\sum}}|A|(|X_i| - n|Z(G)|) - n}, (|A|(|X_1| - |Z(G)|) - &1))^1,\dots,\\
&  (|A|(|X_n| - |Z(G)|) - 1))^1 \} 
\end{align*}
where $X_1,\dots, X_n$ are the distinct centralizers of non-central elements of $G$.  
\end{corollary}

\begin{proof}
It is easy to see that $Z(G\times A) = Z(G)\times A$ and $X_1\times A,  X_2 \times A,\dots, X_n \times A$  are the distinct centralizers of non-central elements of $G\times A$.  Therefore, if $G$ is an  AC-group then $G\times A$ is also an  AC-group. Hence, the result follows from Theorem \ref{AC-group}. 
\end{proof}

%\begin{proof}
%Note that $Z(Q_{4n}) = \{1, y^n\}$. It is not difficult to see that
%$C_{Q_{4n}}(y) = C_{Q_{4n}}(y^i) = \langle y \rangle$ for $1 \leq i \leq 2n - 1$ $(i \ne n)$ and $C_{Q_{4n}}(xy^j) = C_{Q_{4n}}(xy^{j + n}) = \{1, xy^j, xy^{j + n} \}$ for $0 \leq j \leq n - 1$.
%These are the only  centralizers of non-central elements of $Q_{4n}$. Also note that these centralizers are abelian subgroups of  $Q_{4n}$. Therefore
%\[
%\Gamma_{Q_{4n}} = K_{|C_{Q_{4n}}(y)\setminus Z(Q_{4n})|} \sqcup (\underset{j = 1}{\overset{n}{\sqcup}} K_{|C_{Q_{4n}}(xy^j)\setminus Z(Q_{4n})|}).
%\]
%Thus $\Gamma_{Q_{4n}} = K_{2(n - 1)}\sqcup (\underset{j = 1}{\overset{n}{\sqcup}} K_{2})$, since $|Z(Q_{4n})| = 2, |C_{Q_{4n}}(y)| = 2n$ and $|C_{Q_{4n}}(xy^j)| = 4$ for $0 \leq j \leq n - 1$. Therefore, by Theorem \ref{prethm1}, the result follows. 
%\end{proof}

%The following result shows that the   commuting graph of the quasidihedral group is integral.

%The above results show that the commuting graph of a finite non-abelian  group $G$ is integral if $G$ is an AC-group or $G$ is isomorphic to a finite product  AC-groups. %The above corollary shows that the commuting graph of  $G\times A$, where $G$ is a finite non-abelian AC-group and $A$ is any finite avelian group, is integral. 
 Now we compute the spectrum of the  commuting graphs of some particular families of AC-groups. We begin with the well-known family of quasidihedral groups.
\begin{proposition}\label{semid}
The spectrum of the commuting graph of the quasidihedral group $QD_{2^n} = \langle a, b : a^{2^{n-1}} =  b^2 = 1, bab^{-1} = a^{2^{n - 2} - 1}\rangle$, where $n \geq 4$, is given by
\[
\spec(\Gamma_{QD_{2^n}}) = \{(-1)^{2^{n}  - 2^{n - 2} - 3}, 1^{2^{n - 2}}, (2^{n - 1} - 3)^1\}.
\]
\end{proposition}
\begin{proof}
It is well-known that $Z(QD_{2^n}) = \{1, a^{2^{n - 2}}\}$. Also 
\[
C_{QD_{2^n}}(a) = C_{QD_{2^n}}(a^i) = \langle a \rangle \text{ for } 1 \leq i \leq 2^{n - 1} - 1, i \ne 2^{n - 2}
\]
and
\[
C_{QD_{2^n}}(a^jb) = \{1, a^{2^{n - 2}}, a^ib, a^{i + 2^{n - 2}}b \} \text{ for } 1 \leq j \leq 2^{n - 2}
\]
are the only  centralizers of non-central elements of $QD_{2^n}$. Note that these centralizers are abelian subgroups of  $QD_{2^n}$. Therefore, by Lemma \ref{AC-Lem}
\[
\Gamma_{QD_{2^n}} = K_{|C_{QD_{2^n}}(a)\setminus Z(QD_{2^n})|} \sqcup (\underset{j = 1}{\overset{2^{n - 2}}{\sqcup}} K_{|C_{QD_{2^n}}(a^jb)\setminus Z(QD_{2^n})|}).
\]
That is, $\Gamma_{QD_{2^n}} = K_{2^{n - 1} - 2} \sqcup 2^{n - 2} K_2$, since $|C_{QD_{2^n}}(a)| = 2^{n - 1}, |C_{QD_{2^n}}(a^jb)| = 4$ for $1 \leq j \leq 2^{n - 2}$ and  $|Z(QD_{2^n})| = 2$. Hence,  the result follows from  \eqref{prethm1}.
\end{proof}

\begin{proposition}\label{psl}
The spectrum of the commuting graph of the projective special linear group  $PSL(2, 2^k)$, where $k \geq 2$,   is given by
\[
%\spec(\Gamma_{PSL(2, 2^k)}) =
 \{(-1)^{2^{3k} - 2^{2k} - 2^{k + 1} - 2}, (2^k - 1)^{2^{k - 1}(2^k - 1)}, (2^k - 2)^{2^k + 1}, (2^k - 3)^{2^{k - 1}(2^k + 1)}\}. 
\]
\end{proposition}

\begin{proof}
We know that  $PSL(2, 2^k)$ is a non-abelian group of order $2^k(2^{2k} - 1)$ with trivial center. By Proposition 3.21 of \cite{Ab06}, the set of centralizers of non-trivial elements of $PSL(2, 2^k)$ is given by
\[
\{xPx^{-1}, xAx^{-1}, xBx^{-1} : x \in PSL(2, 2^k)\}
\]
where $P$ is an elementary abelian \quad $2$-subgroup and  $A, \quad B$ are  cyclic subgroups of  $PSL(2, 2^k)$ having order $2^k, 2^k - 1$ and $2^k + 1$ respectively. Also the number of conjugates of $P, A$ and $B$ in $PSL(2, 2^k)$ are  $2^k + 1, 2^{k - 1}(2^k + 1)$ and $2^{k - 1}(2^k - 1)$ respectively. Note that $PSL(2, 2^k)$ is a AC-group and so, by Lemma \ref{AC-Lem}, the commuting graph of $PSL(2, 2^k)$ is given by 
\[
%\Gamma_{PSL(2, 2^k)} =
 (2^k + 1)K_{|xPx^{-1}| - 1} \sqcup 2^{k - 1}(2^k + 1)K_{|xAx^{-1}| - 1} \sqcup 2^{k - 1}(2^k - 1)K_{|xBx^{-1}| - 1}.
\]
That is, $\Gamma_{PSL(2, 2^k)} = (2^k + 1)K_{2^k - 1} \sqcup 2^{k - 1}(2^k + 1)K_{2^k - 2} \sqcup 2^{k - 1}(2^k - 1)K_{2^k}$. Hence, the result follows from   \eqref{prethm1}.
%\[
%\{(-1)^{(2^k + 1)(2^k - 1) + 2^{k - 1}(2^k + 1)(2^k - 2) + 2^{k - 1}(2^k - 1)2^k - (2^k + 1 + 2^{k - 1}(2^k + 1) + 2^{k - 1}(2^k - 1))}, (2^k - 2)^{2^k + 1},  (2^k - 3)^{2^{k - 1}(2^k + 1)}, (2^k - 1)^{2^{k - 1}(2^k - 1)} \}
%\]
\end{proof}

\begin{proposition}
The spectrum of the commuting graph of the general linear group  $GL(2, q)$, where $q = p^n > 2$ and $p$ is a prime integer,   is given by
\[
%\spec(\Gamma_{GL(2, q)}) = 
\{(-1)^{q^4 -q^3 - 2q^2 - q}, (q^2 -3q + 1)^{\frac{q(q + 1)}{2}}, (q^2 - q - 1)^{\frac{q(q - 1)}{2}}, (q^2 - 2q)^{q + 1}\}.
\]
\end{proposition}

\begin{proof}
We have $|GL(2, q)| = (q^2 -1)(q^2 - q)$ and $|Z(GL(2, q))| = q - 1$. By Proposition 3.26 of  \cite{Ab06}, the set of centralizers of non-central elements of $GL(2, q)$ is given by
\[
\{xDx^{-1}, xIx^{-1}, xPZ(GL(2, q))x^{-1} : x \in GL(2, q)\}
\]
where $D$ is the subgroup of $GL(2, q)$ consisting  of all diagonal matrices, $I$ is a  cyclic subgroup of $GL(2, q)$ having order $q^2 - 1$  and $P$ is the Sylow $p$-subgroup of $GL(2, q)$ consisting of all upper triangular matrices with $1$ in the diagonal. The orders of  $D$ and $PZ(GL(2, q))$ are  $(q - 1)^2$ and $q(q - 1)$ respectively. Also   the number of conjugates of $D, I$ and $PZ(GL(2, q))$ in $GL(2, q)$  are  $\frac{q(q + 1)}{2}, \frac{q(q - 1)}{2}$ and $q + 1$ respectively. Since $GL(2, q)$ is an AC-group (see Lemma 3.5 of \cite{Ab06}), by Lemma \ref{AC-Lem} we have $\Gamma_{GL(2, q)} =$
\[
 \frac{q(q + 1)}{2}K_{|xDx^{-1}| - q + 1} \sqcup \frac{q(q - 1)}{2}K_{|xIx^{-1}| - q + 1} \sqcup (q + 1)K_{|xPZ(GL(2, q))x^{-1}| - q + 1}.
\]
That is, $\Gamma_{GL(2, q)} = \frac{q(q + 1)}{2}K_{q^2 - 3q + 2} \sqcup \frac{q(q - 1)}{2}K_{q^2 - q} \sqcup (q + 1)K_{q^2 - 2q + 1}$. Hence, the result follows from   \eqref{prethm1}.
\end{proof}

\begin{theorem} \label{order-20}
Let $G$ be a finite group and $\frac{G}{Z(G)} \cong Sz(2)$, where $Sz(2)$ is the Suzuki group presented by $\langle a, b : a^5 = b^4 = 1, b^{-1}ab = a^2 \rangle$. Then
\[
\spec(\Gamma_G) = \{(-1)^{19|Z(G)| - 6}, (4|Z(G)| - 1)^1,  (3|Z(G)| - 1)^5\}.
\]
\end{theorem}
\begin{proof}
We have
\[
\frac{G}{Z(G)} = \langle aZ(G), bZ(G) : a^5Z(G) = b^4Z(G) = Z(G), b^{-1}abZ(G) = a^2Z(G) \rangle.
\]
 Observe that 
\[
\begin{array}{ll}
C_G(a)  &= Z(G)\sqcup aZ(G) \sqcup a^2Z(G)\sqcup a^3Z(G)\sqcup a^4Z(G),\\
C_G(ab) &= Z(G)\sqcup abZ(G) \sqcup a^4b^2Z(G)\sqcup a^3b^3Z(G),\\
C_G(a^2b) &= Z(G)\sqcup a^2bZ(G) \sqcup a^3b^2Z(G)\sqcup ab^3Z(G),\\
C_G(a^2b^3) &= Z(G)\sqcup a^2b^3Z(G) \sqcup ab^2Z(G)\sqcup a^4bZ(G),\\ 
C_G(b) &= Z(G)\sqcup bZ(G) \sqcup b^2Z(G)\sqcup b^3Z(G) \quad \text{ and }\\
C_G(a^3b) &= Z(G)\sqcup a^3bZ(G) \sqcup a^2b^2Z(G)\sqcup a^4b^3Z(G)
\end{array}
\]
are the only centralizers of non-central elements of $G$. Also note that these centralizers are abelian subgroups of $G$. Thus $G$ is an  AC-group. By Lemma \ref{AC-Lem}, we have 
\[
\Gamma_G = K_{4|Z(G)|}\sqcup 5K_{3|Z(G)|}
\]
 since $|C_G(a)| = 5|Z(G)|$ and 
\[
|C_G(ab)| =  |C_G(a^2b)| =  |C_G(a^2b^3)| =  |C_G(b)| =  |C_G(a^3b)| = 4|Z(G)|.
\]
 Therefore, by   \eqref{prethm1}, the result follows.
\end{proof}

\begin{proposition}\label{Hanaki1}
Let $F = GF(2^n), n \geq 2$ and $\vartheta$ be the Frobenius  automorphism of $F$, i. e., $\vartheta(x) = x^2$ for all $x \in F$. Then the spectrum of the commuting graph of the group 
\[
A(n, \vartheta) = \left\lbrace U(a, b) = \begin{bmatrix}
        1 & 0 & 0\\
        a & 1 & 0\\
        b & \vartheta(a) & 1
       \end{bmatrix} : a, b \in F \right\rbrace.
\] 
under matrix multiplication given by $U(a, b)U(a', b') = U(a + a', b + b' + a'\vartheta(a))$ is
\[
\Gamma_{A(n, \vartheta)} = \{(-1)^{(2^n - 1)^2}, (2^n - 1)^{2^n - 1}\}.
\] 
\end{proposition}

\begin{proof}
Note that $Z(A(n, \vartheta)) = \{U(0, b) : b\in F\}$ and so $|Z(A(n, \vartheta))| = 2^n - 1$. Let $U(a, b)$ be a non-central element of $A(n, \vartheta)$. It can be seen that the centralizer of $U(a, b)$ in $A(n, \vartheta)$ is $Z(A(n, \vartheta))\sqcup U(a, 0)Z(A(n, \vartheta))$. Clearly $A(n, \vartheta)$ is an AC-group and so by Lemma \ref{AC-Lem} we have $\Gamma_{A(n, \vartheta)} = (2^n - 1)K_{2^n}$. Hence the result follows by     \eqref{prethm1}.
\end{proof}

\begin{proposition}\label{Hanaki2}
Let $F = GF(p^n)$, $p$ be a prime. Then the spectrum of the commuting graph of the group 
\[
A(n, p) = \left\lbrace V(a, b, c) = \begin{bmatrix}
        1 & 0 & 0\\
        a & 1 & 0\\
        b & c & 1
       \end{bmatrix} : a, b, c \in F \right\rbrace.
\]
under matrix multiplication $V(a, b, c)V(a', b', c') = V(a + a', b + b' + ca', c + c')$ is
\[
\Gamma_{A(n, p)} = \{(-1)^{p^{3n} -2p^{n} -  1},   (p^{2n} - p^n - 1)^{p^n + 1}\}.
\] 
\end{proposition}

\begin{proof}
We have $Z(A(n, p)) = \{V(0, b, 0) : b \in F\}$ and so $|Z(A(n, p))| = p^n$. The centralizers of non-central elements of $A(n, p)$ are given by
\begin{enumerate}
\item If $b, c \in F$ and $c \ne 0$ then the centralizer of $V(0, b, c)$ in $A(n, p)$ is\\ $\{V(0, b', c') : b', c' \in F\}$ having order $|p^{2n}|$. 
\item If $a, b \in F$ and $a \ne 0$ then the centralizer of $V(a, b, 0)$ in $A(n, p)$ is\\ $\{V(a', b', 0) : a', b' \in F\}$ having order $|p^{2n}|$. 
\item If $a, b, c \in F$ and $a \ne 0, c \ne 0$ then the centralizer of $V(a, b, c)$ in $A(n, p)$ is $\{V(a', b', ca'a^{-1}) : a', b' \in F\}$ having order $|p^{2n}|$.
\end{enumerate}
It can be seen that all the centralizers of non-central elements of $A(n, p)$ are abelian. Hence $A(n, p)$ is an AC-group and so 
\[
\Gamma_{A(n, p)} = K_{p^{2n} - p^n}\sqcup K_{p^{2n} - p^n}\sqcup (p^n - 1)K_{p^{2n} - p^n} = (p^n + 1)K_{p^{2n} - p^n}.
\]
Hence the result follows  from   \eqref{prethm1}.
\end{proof}

We would like to mention here that the groups considered in Proposition \ref{Hanaki1}-\ref{Hanaki2} are constructed by Hanaki (see \cite{Han96}). These groups are also considered in \cite{ali00}, in order to compute their numbers of distinct centralizers.

\section{Some applications}
In this section, we show that the  commuting graph of a finite non-abelian group $G$ is  integral if $G$ is not isomorphic to $S_4$ and the commuting graph of $G$ is planar. We also show that the  commuting graph of a finite non-abelian group $G$ is  integral if the commuting graph of $G$   is toroidal. We shall use  the following  results.
% that are obtained in \cite{Nath15}. 
\begin{theorem}\label{main2}
Let $G$ be a finite group such that $\frac{G}{Z(G)} \cong {\mathbb{Z}}_p \times {\mathbb{Z}}_p$, where $p$ is a prime integer. Then 
\[
\spec(\Gamma_G) = \{(-1)^{(p^2 - 1)|Z(G)| - p - 1}, ((p - 1)|Z(G)| - 1)^{p + 1}\}.
\]
\end{theorem}
\begin{proof}
The result follows from Theorem \ref{AC-group} noting that $G$ is an AC-group with $p+1$ distinct centralizers of non-central elements and all of them have order $p|Z(G)|$.
\end{proof}

\begin{proposition}\label{main05}
Let $D_{2m} = \langle a, b : a^m = b^{2} = 1, bab^{-1} = a^{-1} \rangle$ be  the dihedral group of order $2m$, where $m > 2$. Then 
\[
\spec(\Gamma_{D_{2m}}) = \begin{cases}
&\{(-1)^{m - 2}, 0^m, (m - 2)^1\} \text{ if $m$ is odd}\\
&\{(-1)^{\frac{3m}{2} - 3}, 1^{\frac{m}{2}}, (m - 3)^1\} \text{ if $m$ is even}. 
\end{cases} 
\]  
\end{proposition}
\begin{proof}
Note that $D_{2m}$ is a non-abelian AC-group. If $m$ is even then $|Z(D_{2m})| = 2$ and $D_{2m}$ has $\frac{m}{2} + 1$ distinct centralizers of non-central elements. Out of these centralizers one has order $m$ and the rests have order $4$. Therefore $\Gamma_{D_{2m}} = K_{m - 2} \sqcup \frac{m}{2}K_2$. If $m$ is odd then  $|Z(D_{2m})| = 1$ and $D_{2m}$ has $m + 1$ distinct centralizers of non-central elements. In this case, one centralizer has order $m$ and the rests have order $2$. Therefore $\Gamma_{D_{2m}} = K_{m - 1} \sqcup mK_1$. Hence the result follows from  \eqref{prethm1}. 
\end{proof}

\begin{proposition}\label{q4m}
The spectrum of the commuting graph of the generalized quaternion group $Q_{4n} = \langle x, y : y^{2n} = 1, x^2 = y^n,yxy^{-1} = y^{-1}\rangle$, where $n \geq 2$, is given by
\[
\spec(\Gamma_{Q_{4n}}) = \{(-1)^{3n - 3},  1^n, (2n - 3)^1\}.
\]
\end{proposition}
\begin{proof}
Note that $Q_{4n}$ is a non-abelian AC-group with $n + 1$ distinct centralizers of non-central elements. Out of these centralizers one has order $2n$ and the rests have order $4$. Also $|Z(Q_{4n})| = 2$. Therefore $\Gamma_{Q_{4n}} = K_{2n - 2} \sqcup nK_2$. Hence the result follows from  \eqref{prethm1}.   
\end{proof}

As an application of Theorem \ref{main2} we have the following lemma.
\begin{lemma}\label{order16}
Let $G$ be a group isomorphic to any of the following groups
\begin{enumerate}
\item ${\mathbb{Z}}_2 \times D_8$
\item ${\mathbb{Z}}_2 \times Q_8$
\item $M_{16}  = \langle a, b : a^8 = b^2 = 1, bab = a^5 \rangle$
\item ${\mathbb{Z}}_4 \rtimes {\mathbb{Z}}_4 = \langle a, b : a^4 = b^4 = 1, bab^{-1} = a^{-1} \rangle$
\item $D_8 * {\mathbb{Z}}_4 = \langle a, b, c : a^4 = b^2 = c^2 =  1, ab = ba, ac = ca, bc = a^2cb \rangle$
\item $SG(16, 3)  = \langle a, b : a^4 = b^4 = 1, ab = b^{-1}a^{-1}, ab^{-1} = ba^{-1}\rangle$.
\end{enumerate}
Then $\spec(\Gamma_G) = \{(-1)^9, 3^3\}$.
\end{lemma}
\begin{proof}
If $G$ is isomorphic to any of the above listed  groups, then $|G| = 16$ and $|Z(G)| = 4$. Therefore, $\frac{G}{Z(G)} \cong {\mathbb{Z}}_2 \times {\mathbb{Z}}_2$. Thus the result follows from Theorem~\ref{main2}.
\end{proof}

The next lemma is also useful in this section.
\begin{lemma}\label{order-pq}
Let $G$ be a non-abelian group of order $pq$, where $p$ and $q$ are primes with $p\mid (q - 1)$. Then
\[
\spec(\Gamma_G) = \{(-1)^{pq - q -1}, (p - 2)^q, (q - 2)^1\}.
\]
\end{lemma}

\begin{proof}
It is easy to see that $|Z(G)| = 1$ and $G$ is an AC-group. Also the centralizers of non-central elements of $G$ are precisely the Sylow subgroups of $G$. The number of Sylow $q$-subgroups and Sylow $p$-subgroups of $G$ are one and $q$ respectively. Therefore, by Lemma \ref{AC-Lem} we have $\Gamma_G = K_{q-1} \sqcup qK_{p - 1}$. Hence, the result follows from   \eqref{prethm1}.  
\end{proof}

Now we state and proof  the main results of this section. 
\begin{theorem}
Let $\Gamma_G$ be the commuting graph of a finite non-abelian  group $G$.  If $G$ is not isomorphic to $S_4$ and $\Gamma_G$  is planar   then  $\Gamma_G$ is integral. 
\end{theorem}

\begin{proof}
By Theorem 2.2 of \cite{AF14} we have that $\Gamma_G$ is planar if and only if $G$ is isomorphic to either $D_6, D_8, D_{10}, D_{12}, Q_8, Q_{12}, {\mathbb{Z}}_2 \times D_8, {\mathbb{Z}}_2 \times Q_8, M_{16}, {\mathbb{Z}}_4 \rtimes {\mathbb{Z}}_4, D_8 * {\mathbb{Z}}_4, SG(16, 3), A_4,$ $A_5, S_4, SL(2, 3)$ or $Sz(2) = \langle a, b : a^5 = b^4 = 1, b^{-1}ab = a^3\rangle$.

If $G \cong D_6, D_8, D_{10}$ or $D_{12}$ then by Proposition \ref{main05}, one may conclude that  $\Gamma_G$ is integral. If $G \cong Q_8$ or $Q_{12}$ then  by Proposition \ref{q4m},    $\Gamma_G$ becomes integral. If $G \cong {\mathbb{Z}}_2 \times D_8, {\mathbb{Z}}_2 \times Q_8, M_{16}, {\mathbb{Z}}_4 \rtimes {\mathbb{Z}}_4, D_8 * {\mathbb{Z}}_4$ or $SG(16, 3)$ then by Lemma \ref{order16}, $\Gamma_G$ becomes integral.

If $G \cong A_4 = \langle a, b : a^2 = b^3 = (ab)^3 = 1\rangle$ then the distinct centralizers of non-central elements of $G$ are $C_{G}(a) = \{1, a, bab^2, b^2ab\}, C_{G}(b) =\{1, b, b^2\}$, $C_{G}(ab) = \{1, ab, b^2a\}, C_{G}(ba) = \{1, ba, ab^2\}$ and $C_{G}(aba) = \{1, aba, bab\}$. Note that these centralizers are abelian subgroups of $G$. Therefore, $\Gamma_{G} = K_3 \sqcup 4K_2$ and 
\[
\spec(\Gamma_{G}) = \{(-1)^6, 2^1, 1^4\}.
\]
 Thus $\Gamma_{G}$ is integral.

If $G \cong Sz(2)$ then by Theorem \ref{order-20}, we have
\[
\Gamma_G = \{(-1)^{13}, (3)^1, (2)^5\}.
\]
Hence, $\Gamma_G$ is integral.

If $G$ is isomorphic to 
\begin{align*}
 SL(2, 3) = \langle a, b, c : a^3 = b^4 = 1, &b^2 = c^2,\\
  &c^{-1}bc = b^{-1}, a^{-1}ba = b^{-1}c^{-1}, a^{-1}ca = b^{-1} \rangle
\end{align*}
then $Z(G) = \{1, b^2\}$.  It can be seen that
%\[
%C_G(b)  = \{1, b, b^2, b^3\} = \langle b\rangle,  C_G(c)  = \{1, c, c^2, c^3\} = \langle c\rangle,
%C_G(bc)  = \{1, b^2, bc, cb\} = \langle bc\rangle 
%\]

\[
\begin{array}{ll}
C_G(b)      &= \{1, b, b^2, b^3\} = \langle b\rangle,\\
C_G(c)      &= \{1, c, c^2, c^3\} = \langle c\rangle,\\
C_G(bc)     &= \{1, b^2, bc, cb\} = \langle bc\rangle,\\
C_G(a^2b^2) &= \{1, b^2, a, a^2, a^2b^2, ab^2\}  = \langle a^2b^2 \rangle,\\
C_G(ac)     &= \{1, b^2, ac, ca^2, a^2bc, ab^2c\} = \langle ac \rangle,\\
C_G(ca)     &= \{1, b^2, ca, a^2c, ba^2, ab\} = \langle ca \rangle \quad \text{ and }\\ 
C_G(a^2b)   &= \{1, b^2, a^2b, ba, b^3a, (ba)^2\} = \langle a^2b \rangle 
\end{array}
\]
are the only distinct  centralizers of non-central elements of $G$.  Note that these centralizers are abelian subgroups of $G$. Therefore, $\Gamma_G = 3K_2 \sqcup 4K_4$ and 
\[
\spec(\Gamma_G) = \{(-1)^{15}, 1^3, 3^4\}.
\]
 Thus $\Gamma_{G}$ is integral.

%If $G \cong A_5$ then, as noted in \cite{ali05}, the distinct centralizers of $G$ are Sylow subgroups of $G$.  The numbers  of Sylow $2$-subgroups, Sylow $3$-subgroups and Sylow $5$-subgroups of  $G$ are $5, 10$ respectively. Since any Sylow subgroup of $A_5$ is abelian and  any two sylow subgroups intersects only at $Z(A_5) = \{(1)\}$ we have $\Gamma_G = 10K_2 \sqcup 5K_3 \sqcup 6K_4$. So,

If $G \cong A_5$ then by Proposition \ref{psl}, we have
\[
\spec(\Gamma_G) = \{(-1)^{38}, 1^{10}, 2^{5}, 3^6\}
\] 
noting that $PSL(2, 4) \cong A_5$.  Thus $\Gamma_{G}$ is integral.

Finally, if $G \cong S_4$ then
% using GAP \cite{gapo8} 
it can be seen that the  characteristic polynomial of $\Gamma_G$ is $(x - 1)^7(x + 1)^{10}(x^2 - 5)^2(x^2 - 3x - 2)$ and so 
\[
\spec(\Gamma_G) = \left\lbrace 1^7, (-1)^{10}, (\sqrt{5})^2, (-\sqrt{5})^2, \left(\frac{3 + \sqrt{17}}{2}\right)^1, \left(\frac{3 - \sqrt{17}}{2}\right)^1 \right\rbrace.
\] 
Hence, $\Gamma_G$ is not integral. This completes the proof.
\end{proof}

In \cite[Theorem 2.3]{AF14}, Afkhami et al. have classified all finite non-abelian groups whose commuting graphs are toroidal. Unfortunately, the  statement of Theorem 2.3 in \cite{AF14} is printed incorrectly.  We list the correct version of \cite[Theorem 2.3]{AF14} below, since we are going to use it.
\begin{theorem}\label{toroidal}
Let $G$ be a finite non-abelian group. Then $\Gamma_G$ is toroidal if and only if $\Gamma_G$ is projective if and only if $G$ is isomorphic to either $D_{14}, D_{16}, Q_{16}$, $QD_{16},   D_6 \times {\mathbb{Z}}_3,   A_4 \times {\mathbb{Z}}_2$ or ${\mathbb{Z}}_7 \rtimes {\mathbb{Z}}_3$. 
\end{theorem}

\begin{theorem}
Let $\Gamma_G$ be the commuting graph of a finite non-abelian  group $G$.  Then $\Gamma_G$ is integral if  $\Gamma_G$    is toroidal. 
\end{theorem}
\begin{proof}
By Theorem \ref{toroidal}  we have that $\Gamma_G$ is toroidal if and only if $G$ is isomorphic to either $D_{14}, D_{16}, Q_{16}, QD_{16},   D_6 \times {\mathbb{Z}}_3,   A_4 \times {\mathbb{Z}}_2$ or ${\mathbb{Z}}_7 \rtimes {\mathbb{Z}}_3$.

If $G \cong D_{14}$ or $D_{16}$ then by Proposition \ref{main05}, one may conclude that  $\Gamma_G$ is integral. If $G \cong Q_{16}$ then  by Proposition \ref{q4m},    $\Gamma_G$ becomes integral. If $G \cong QD_{16}$ then by Proposition \ref{semid}, $\Gamma_G$ becomes integral. If 
$G \cong {\mathbb{Z}}_7 \rtimes {\mathbb{Z}}_3$ then $\Gamma_G$ is integral, by Lemma \ref{order-pq}.  If $G$ is isomorphic to  $D_6\times {\mathbb{Z}}_3$ or $A_4\times {\mathbb{Z}}_2$ then $\Gamma_G$ becomes integral by Corollary \ref{AC-cor}, since   $D_6$ and $A_4$ are AC-groups. This completes the proof. 
\end{proof}

We shall conclude the paper with the following result.

\begin{proposition}
Let $\Gamma_G$ be the commuting graph of a finite non-abelian  group $G$. Then  $\Gamma_G$ is integral if the complement of  $\Gamma_G$  is planar.
\end{proposition}
\begin{proof}
If  the complement of  $\Gamma_G$  is planar then by Proposition 2.3 of \cite{Ab06} we have that $G$ is isomorphic to either $D_6, D_8$ or $Q_8$. If $G \cong D_6$ or $D_8$ then by Proposition \ref{main05},  $\Gamma_G$ is integral.  If $G \cong Q_8$ then  by Proposition \ref{q4m},    $\Gamma_G$ becomes integral. This completes the proof.
\end{proof}


\begin{thebibliography}{}
% etc


\bibitem{Ab06}
A. Abdollahi, S. Akbari and H. R. Maimani, Non-commuting graph of a group, {\em  J. Algebra}, {\bf 298}, 468--492 (2006).

%\bibitem{ajH07}
%A. Abdollahi, S. M. Jafarain and A. M. Hassanabadi, Groups with specific number of centralizers, {\em Houston J. Math.}, {\bf 33}(1) (2007), 43--57.

%\bibitem{aV09}
%A. Abdollahi and E. Vatandoost, Which Cayley graphs are integral? {\em Electron. J. Combin.} {\bf 16} (1) R122, 17 pp., 2009.

\bibitem{AF14}
M. Afkhami, M. Farrokhi  D. G. and K. Khashyarmanesh, Planar, toroidal, and projective commuting and non-commuting  graphs, {\em Comm. Algebra}, {\bf 43}(7), 2964--2970 (2015).


%\bibitem{anb09}
%O. Ahmadi, N. Alon, I. F. Blake and I. E. Shparlinski, Graphs with integral spectrum, {\em Linear Algebra Appl.}, {\bf 430}(1) (2009), 547-552.

\bibitem{amr06}
 S. Akbari, A. Mohammadian, H. Radjavi and P. Raja, On the diameters of commuting graphs, {\em Linear Algebra  Appl.},  {\bf 418}, 161--176 (2006). 

%\bibitem{al12}
%R. C. Alperin and B. L. Peterson, Integral sets and Cayley graphs of finite groups, {\em Electron. J. Combin.}, 19 (2012), \#P44.


\bibitem{ali00}
A. R. Ashrafi, On finite groups with a given number of centralizers, {\em Algebra Colloq.}, {\bf 7}(2), 139--146 (2000).

%\bibitem{ali05}
%A. R. Ashrafi and B. Taeri, On finite groups with a certain number of centralizers, {\em J. Appl. and Computing}, {\bf 17}(1-2) (2000), 217--227.

 



%\bibitem{bG94}
%S. M. Belcastro and G. J. Sherman,  Counting centralizers in finite groups, {\em Math. Magazine}, {\bf 67} (5) (1994), 366--374.








\bibitem{bCrS03} 
K. Bali$\acute{\rm n}$ska, D. Cvetkovi$\acute{\rm c}$, Z. Radosavljevi$\acute{\rm c}$, S. Simi$\acute{\rm c}$ and D. Stevanovi$\acute{\rm c}$, A survey on integral graphs, {\em Univ. Beograd. Publ. Elektrotehn. Fak. Ser. Mat.} {\bf 13}, 42--65 (2003).


\bibitem{bbhR09}
 C. Bates, D. Bundy, S. Hart and P. Rowley, A Note on Commuting Graphs for Symmetric Groups, {\em Electron. J. Combin.} {\bf 16}, 1--13 (2009).

%\bibitem{Br08}
% A. E. Brouwer, Small integral trees, {\em Electron. J. Combin.} {\bf 15} (2008), Note 1, 1--8.

%\bibitem{bH14} 
 %A. E. Brouwer and W. H. Haemers, The integral trees with spectral
%radius 3, {\em Linear  Algebra Appl.}, (to appear).

\bibitem{das13}
A. K. Das and D. Nongsiang, On the genus of the commuting  graphs of finite non-abelian groups, preprint, arXiv:1311.6324vl



\bibitem{Han96}
A. Hanaki, A condition of lengths of conjugacy classes and character degree, {\em Osaka J. Math} {\bf 33}, 207--216 (1996).


\bibitem{hS74}
F. Harary and A. J. Schwenk, Which graphs have integral spectra?,
\textit{Graphs and Combin.}, Lect. Notes Math., Vol 406, Springer-Verlag,
Berlin,  45--51 (1974).


%\bibitem{iV07}  
%G. Indulal and A. Vijayakumar, Some new integral graphs, {\em Appl. Anal. Discrete Math.} {\bf 1} (2007), 420--426.


\bibitem{iJ07}
 A. Iranmanesh and A. Jafarzadeh, Characterization of finite groups by their commuting graph, {\em Acta Mathematica Academiae Paedagogicae Nyiregyhaziensis}, {\bf 23}(1), 7--13 (2007).

\bibitem{iJ08}  
 A. Iranmanesh and A. Jafarzadeh, On the commuting graph associated with the symmetric and alternating groups, {\em J. Algebra Appl.}, {\bf 7}(1), 129--146 (2008).

\bibitem{mP13} 
 G. L. Morgan and C. W. Parker, The diameter of the commuting graph of a finite group with trivial center, {\em J. Algebra} {\bf 393}(1), 41--59 (2013).
 
%\bibitem{Nath15} 
%R. K. Nath and J. Dutta, Finite groups whose commuting graphs are integral, submitted for publication.

%\bibitem{Nd15} 
%R. K. Nath and J. Dutta, Spectrum of commuting graphs of some finite rings, preprint.

\bibitem{par13} 
 C. Parker, The commuting graph of a soluble group, {\em Bull. London Math. Soc.}, {\bf 45}(4), 839--848 (2013).
 
\bibitem{Roc75}
 D. M. Rocke,  $p$-groups with abelian centralizers, \emph{Proc. London Math. Soc.}  {\bf 30}(3), 55--75 (1975). 

%\bibitem{So07}  
%W. So, Integral circulant graphs, {\em Discrete Math.} {\bf 306} (2006), 153--158.

%\bibitem{saF07} 
% D. Stevanovi$\acute{\rm c}$, N. M. M. de Abreu, M. A. A. de Freitas and R. Del-Vecchio, Walks and regular integral graphs, {\em Linear Algebra Appl.} {\bf 423} (2007), 119--135.

%\bibitem{wL05}  
% L. Wang and X. Li, Integral trees with diameters 5 and 6, {\em Discrete
%Math.} {\bf 297} (2005), 128--143.


%\bibitem{wLH05} 
%L. Wang, X. Li and C. Hoede, Two classes of integral regular graphs,
%{\em Ars Combin.} {\bf 76} (2005), 303--319.

%\bibitem{gapo8}
%The GAP Group, GAP - Groups, Algorithms and Programming, version 4.4.12, 2008. http://www.gap-system.org


\end{thebibliography}
\end{document}